\documentclass[a4paper]{article}
\usepackage{amssymb}
\usepackage{amsthm}
\theoremstyle{plain}\newtheorem{thm}{Theorem}
\theoremstyle{plain}
\theoremstyle{plain}
\newtheorem{lem}[thm]{Lemma}
\newtheorem{cor}[thm]{Corollary}
\theoremstyle{definition}

\def\gz{|G:\mathbf{Z}(G)|}
\def\gzz{|G:\mathbf{Z}_2(G)|}
\def\cgg{C_G(G')}
\def\zg{\mathbf{Z}(G)}
\def\zzg{\mathbf{Z}_2(G)}

\usepackage{amsmath}
\begin{document}
\title{ON FINITE GROUPS WHOSE DERIVED SUBGROUP HAS BOUNDED RANK}
\author{K. PODOSKI  and B. SZEGEDY}
\date{}
\maketitle

\begin{abstract}Let $G$ be a finite group with derived subgroup of rank $r$. We
prove that $\gzz\leq |G'|^{2r}$. Motivated by the results of I. M.
Isaacs in \cite{isa} we show that if $G$ is capable then $\gz\leq
|G'|^{4r}$. This answers a question of L. Pyber. We prove that if
$G$ is a capable $p$-group then the rank of $G/\mathbf{Z}(G)$ is
bounded above in terms of the rank of $G'$.
\end{abstract}\bigskip

\section{Introduction}
Let $G$ be a finite group. We denote by $d(G)$ the minimal number of
generators of $G$. The rank of G denotes the minimal number
$r=\mathrm{rk}(G)$ such that every subgroup of $G$ can be generated
by $r$ elements. It is clear that $d(G)\leq
\mathrm{rk}(G)\leq\log_2|G|$. The importance of the concept of rank
is underlined by results due to A. Lubotzky and A. Mann (see
\cite{prop} ).  There are various upper bounds in asymptotic group
theory in which the logarithm of the order of certain subgroups
occurs. Our goal is to replace this logarithm by the rank of the
same subgroups. Related to a famous theorem of Schur, Wiegold
\cite{www} proved that if $\gz=n$ then $|G'|\leq
n^{\frac{1}{2}\log_2n}$. R. Guralnick \cite{gg} gave a simple proof
for this fact. Following a similar argument we obtain the following.

\begin{thm}If $G$ is a finite group with
$\mathrm{rk}(G/Z(G))=r$. Then $$|G'|\leq |G:Z(G)|^{r+1}.$$
\end{thm}

The extra-special groups show that there is no upper bound for the
index of the centre in terms of the order of the derived subgroup.
However, P. Hall (see \cite{rrr} p.423) observed that $\gzz$ is
bounded from above in terms of $|G'|$ (where $\zzg$  denotes the
second member of the upper central series of $G$). We proved in
\cite{psz} that
$$\gzz\leq |G'|^{2\log_2|G'|}.$$ L. Pyber \cite{pyb} asked whether
there is a constant $c$ such that $\gzz\leq
|G'|^{c\cdot\mathrm{rk}(G')}$. In the present paper we prove that

\begin{thm}
If $G$ is a finite group and $\mathrm{rk}(G')=r$ then $$\gzz\leq
|G'|^{2r}.$$
\end{thm}

A group $H$ is said to be capable if there exists some $G$ such
that $G/Z(G)$ is isomorphic to $H$. I. M. Isaacs \cite{isa} proved
that if $H$ is a capable group then $|H:Z(H)|$ is bounded above in
terms of $|H'|$, or equivalently, if $G$ is an arbitrary finite
group then $\gzz$ is bounded above in terms of $|G':G'\cap\zg|$,
but he has not given an explicit bound. We proved in \cite{cap}
that if $|G':G'\cap\zg|=n$ then
$$\gzz\leq n^{2\log_2n}.$$
It was proved in \cite{isa} that if $H'$ is cyclic in a capable
group $H$ and all the order $4$ elements in $H'$ are central then
$|H:\mathbf{Z}(H)|\leq|H'|^2$. It was shown later in \cite{cap} that
the condition on the order $4$ elements can be omitted. Motivated by
these results, L. Pyber \cite{pyb} asked if there is a constant $c$
such that $$|H:\mathbf{Z}(H)|\leq n^{c\cdot\mathrm{rk}(H')}$$ in
every capable group $H$. We answer this question in the affirmative
by showing that
\begin{thm}
If $G$ is a finite group and $\mathrm{rk}(G'/G'\cap\zg)=r$ then
$$\gzz\leq |G'/G'\cap\zg|^{4r}.$$
\end{thm}

\begin{cor}If $H$ is a finite capable group and $\mathrm{rk}(H')=r$ then
$$|H:Z(H)|\leq |H'|^{4r}.$$
\end{cor}

 We obtain a  better bound in the case when $\zg=1$. The crucial observation
is as follows.

\begin{thm}If $G$ is a finite group with $Z(G)=1$ then $C_G(G')\leq G'$.
\label{triv}
\end{thm}

As a consequence we obtain the following.
\begin{thm}If $G$ be a finite group with $Z(G)=1$ and $d=d(G')$ then
$|G|\leq |G'|^{d+1}$.
\end{thm}

Finally we show the following bound for the rank of
$G/\mathbf{Z}_2(G)$ in finite $p$-groups.

\begin{thm}\label{pcsop} If $G$ is a finite $p$-group and $\mathrm{rk}(G'/G'\cap\zg)=r$
then $$\mathrm{rk}(G/\mathbf{Z}_2(G))\leq \frac{1}{2}(13r^2-r).$$
\end{thm}

\section{Proof of the results}
 We shall use the following well-known fact \cite{prop}.
\begin{lem}If $P=\langle a_1,a_2,\dots, a_d,Z(P)\rangle$ is a $p$-group
then every element of $P'$ is equal to a product of the form
$[x_1,a_1][x_2,a_2]\dots [x_d,a_d]$ with $x_1,x_2,\dots,x_d\in
P$.\label{pcsop}
\end{lem}

\begin{proof}[Proof of Theorem 1]
 Let $P$ be a Sylow-$p$ subgroup of $G$. It is well-known
(see \cite{gla}, p. 9 ) that $G'\cap P\cap Z(G)=P'\cap Z(G).$ It follows that
$$|G'\cap Z(G)|_p=|G'\cap P\cap Z(G)|=|P'\cap Z(G)|\leq |P'|.$$
By using Lemma \ref{pcsop} we obtain that  $$|P'|\leq
\prod_{i=1}^d|\{[p,a_i]\mid p\in
P\}|=\prod_{i=1}^d|P:C_P(a_i)|\leq |P:Z(P)|^d\leq$$ $$ |P
:Z(P)|^r\leq |P:P\cap Z(G)|^r=|G:Z(G)|_p^r.$$ Consequently,
$$|G'|_p\leq|G:Z(G)|_p|G'\cap Z(G)|_p\leq |G:Z(G)|_p^{r+1}.$$
Since this inequality holds for all prime divisors of $|G|$ the
proof is complete.

\end{proof}

\begin{lem} $Z_2(G)\leq  C_G(G')$ and $[C_G(G'),C_G(G')]\leq
Z(G)$. In particular $\cgg$ is nilpotent of class $2$ and every
Sylow-$p$ subgroup of $\cgg$ is normal in $G$. \label{komcent}
\end{lem}

\begin{proof}Using the Three Subgroup Lemma we obtain that $$[G, G,
Z_2(G)]\leq[Z_2(G) , G , G]=[G, Z_2(G) , G ]=1$$ and $$[C_G(G'),
C_G(G'), G]\leq [C_G(G'), G, C_G(G')]=[G, C_G(G'), C_G(G')]=1.$$
Let $P$ be a Sylow-$p$ subgroup of $\cgg$. Since
$\cgg\triangleleft G$ and $P$ is characteristic in $\cgg$ we have
that $P\vartriangleleft G$.
\end{proof}

\begin{lem} Let $H,K\leq G$ two subgroups such that $H$ can be
generated by $d$ elements and $K$ is normal in $G$. Then
$|K:C_K(H)|\leq |G'\cap K|^d$.\label{ketall}
\end{lem}

\begin{proof}Let $x$ be an arbitrary element of $G$. It is easy to
see that $|K:C_K(x)|=|\{[x,k]\mid k\in K\}|$. Using that $K$ is
normal in $G$ we have that $\{[x,k]\mid k\in K\}\subset K\cap G'$.
It follows that $|K:C_K(x)|\leq |G'\cap K|$.

 Let $x_1,x_2,...,x_d$ be a generating system of $H$.

$$|K:C_K(H)|\leq\prod_{i=1}^d|K:C_K(x_i)|\leq|G'\cap K|^d.$$
\end{proof}

\begin{cor}
If $G$ is a finite group and $d(G')=d$ then $|G:C_G(G')|\leq
|G'|^d$. \label{kci}
\end{cor}

\begin{proof}We apply Lemma \ref{ketall} for $H=G'$ and $K=G$.
\end{proof}

\begin{proof}[Proof of Theorem 5]
Let $P$ be a  Sylow $p$-subgroup of $C_G(G')$. Using Lemma
\ref{komcent} we have that $P\vartriangleleft G$. It follows that
the Abelian group $G/G'$ acts on $P$ under conjugation. Let $H$ be
the Sylow $p$-subgroup of $G/G'$ and $Q$ be its complement.
Applying  Fitting's Lemma (\cite{gor} p. 180) we obtain that $P=[P,Q]\times
C_P(Q)$. Assume that $ |C_P(Q)|>1$. Since the $p$-group $H$ acts
on the $p$-group $C_P(Q)$ we have that $L=C_P(H)\cap C_P(Q)$ is
nontrivial therefore $L$ is central in $G$, which is a
contradiction. Hence $P=[P,Q]\leq G'$. This holds for each Sylow $p$-subgroup
of $C_G(G')$ which completes the proof.
\end{proof}

\begin{proof}[Proof of Theorem 6]
The result  follows immediately from Corollary \ref{kci} and
Theorem \ref{triv}.
\end{proof}

\begin{lem}
Let $A$ be a finite Abelian $p$-group of rank $r$. Let
$\mathcal{S}$ be a collection of subgroups of $A$ such that
$\bigcap\mathcal{S}=\{1\}$. Then there exists a subset
$\mathcal{R}$ of $\mathcal{S}$ such that $|\mathcal{R}|\leq r$ and
$\bigcap\mathcal{R}=\{1\}$.\label{abel}

\end{lem}

\begin{proof}Let $S=\mathrm{Soc}(A)$ be the subgroup generated by the elements
of order $p$. Since $S$ is an elementary Abelian $p$-group of rank
$r$ we can construct a descending chain $S=S_0>S_1>\dots >S_l={1}$
where $l\leq r$ and
$$S_i=S_{i-1}\cap H_{i} \mbox{ for some }  H_{i}\in\mathcal{S}\ \ \ \ (1\leq i\leq
l)$$ Let $\mathcal{R}=\{H_1,H_2,\dots,H_l\}$. We have that
$|\mathcal{R}|=l\leq r$ and that the intersection of
$\bigcap\mathcal{R}$ and $S$ is trivial. Since
$\bigcap\mathcal{R}$ does not contain any element of order $p$ it
follows that $\bigcap\mathcal{R}=\{1\}$.
\end{proof}

\begin{lem}
Let $G$ be a finite group with $Z=G'\cap\zg$ and
$\mathrm{rk}(G'/Z)=r$. Then $|\cgg:\zzg|\leq
|G':Z|^r$.\label{also}
\end{lem}

\begin{proof}
Let $p$ be a prime divisor of $|\cgg|$ and let $P$ be a Sylow-$p$
subgroup of $\cgg$. Using Lemma \ref{komcent}, $P$ is normal in
$G$. It is clear that $P\cap G'$ is an Abelian group. Now
$$\bigcap_{x\in G}C_{P\cap G'}(x)=P\cap Z.$$ Applying Lemma
\ref{abel} with $$A=P\cap G'/P\cap Z,\ \ \ \ \ \
\mathcal{S}=\{C_{P\cap G'}(x)/P\cap Z\mid x\in G\}$$ we obtain
that there exist elements $x_1,x_2,...,x_l$ with
$l\leq\mathrm{rk}(P\cap G'/P\cap Z)\leq r$ such that
$$\bigcap_{1\leq i\leq l }C_{P\cap G'}(x_i)=P\cap Z.$$

Let $T=\langle x_1,x_2,...,x_l\rangle$ and $M/Z=C_{G/Z}(TZ/Z)$. Now
we apply the Three Subgroup Lemma for $P\cap M$ , $T$ and $G$. Since
$P\cap M\leq \cgg$ we have that $[G,T,M\cap P]=1$. Using that
$$[M\cap P,T]\leq [M,T]\leq Z$$ we obtain that $[M\cap P,T,G]=1$. It
follows that $[M\cap P,G,T]=1$ and so $$[M\cap P,G]\leq C_G(T)\cap
P\cap G'\leq Z.$$ This implies that $M\cap P\leq \zzg\cap P$.

Let us denote $\hat{G'}$, $\hat{T}$, $\hat{M}$ and $\hat{P}$ the
images of $G'$, $T$, $M$ and $P$ in the factor group $G/Z$.
Applying the second statement of Lemma \ref{ketall} for
$\hat{K}=\hat{P}$ and $\hat{H}=\hat{T}$, and using that $Z\leq M$,
one gets that $$|P:P\cap \zzg|\leq|P:P\cap
M|=|PM:M|=|\hat{P}\hat{M}:\hat{M}|=$$ $$=|\hat{P}:\hat{P}\cap
\hat{M}|\leq |\hat{G'}:\hat{P}\cap\hat{G'}|^r\leq (n_p)^r$$ where
$n_p$ denotes the $p$-part of $|G':Z|$.

 Let $P_1$,$P_2$,...,$P_t$
be the unique Sylow subgroups of $\cgg$ corresponding to the prime
divisors $p_1$,$p_2$,...,$p_t$ of $|\cgg|$. By embedding the Sylow
subgroups of $\zzg$ into Sylow-$p$ subgroups of $\cgg$ we obtain
that $$|\cgg:\zzg|=\prod_{1\leq i\leq t}|P_i:P_i\cap \zzg|\leq
\prod_{1\leq i\leq t}n_{p_i}^r=|G':Z|^r.$$\end{proof}

\begin{proof}[Proof of Theorem 2]Since $$\gzz=|G:\cgg||\cgg:\zzg|$$ and $$rk(G'/Z)\leq
rk(G')=r$$ Corollary \ref{kci} and Lemma \ref{also} complete the
proof.
\end{proof}

\begin{lem}
Let $G$ be a finite group, let $D=\{g\in G\mid [g,G']\subseteq
\zg\}$ and let $P$ be a Sylow-$p$ subgroup of $D$. Then
$G'/C_{G'}(P)$ is a $p$-group. \label{bilin}
\end{lem}

\begin{proof} Since $D/\zg$ is the centralizer of the commutator
subgroup of $G/\zg$ Lemma \ref{komcent} yields that $D/\zg$ is
nilpotent therefore $D$ is nilpotent. It follows that $P$ is
normal in $G$. Let $q\neq p$ be a prime divisor of $G'$ and let
$Q$ be a Sylow-$q$ subgroup of $G'$. Applying Fitting's Lemma
(\cite{gor} p. 180) for the action of $Q$ on $P$ we get that
$P=[Q,P]C_P(Q)$. Since
$$[Q,P]\leq [G',D]\cap P\leq \zg\cap P$$ we obtain that $P=C_P(Q)$
thus $Q\leq C_{G'}(P)$. Now we have that $C_{G'}(P)$ is a normal
subgroup of $G'$ which contains all Sylow-$q$ subgroups of $G'$
for all $q\neq p$ primes.
\end{proof}

We will make use of the following well known fact.

\begin{lem} If $F$ is a $p$-group with $d(F)=d$ then any
generating set of $F$ contains a subset of size $d$ which is a
generating set as well.\label{genp}
\end{lem}

\begin{lem}
Let $G$ be a finite group, $D=\{g\in G\mid [g,G']\subseteq \zg\}$,
$Z=G'\cap \zg$ and $\mathrm{rk}(G'/Z)=r$. Then $|D:\cgg|\leq
|G':Z|^r$.\label{szivas}
\end{lem}

\begin{proof} Let $p$ be a prime divisor of $|D|$ and let $P$ be the
Sylow-$p$ subgroup of $D$. According to Lemma \ref{bilin},
$G'/C_{G'}(P)$ is a $p$-group. Since $G'$ is generated by the
commutator words $[x,y]~(x,y\in G)$ the factor $G'/C_{G'}(P)$ is
generated by their images under the natural homomorphism. Using
that $G'/C_{G'}(P)$ is a $p$-group and
$\mathrm{rk}(G'/C_{G'}(P))\leq \mathrm{rk}(G'/Z)=r$, Lemma
\ref{genp} guarantees that there exist elements
$x_1,y_1,x_2,y_2,...,x_r,y_r\in G$ such that $$\langle
[x_1,y_1],[x_2,y_2],...,[x_r,y_r],C_{G'}(P)\rangle=G'.$$

Let $T=\langle x_1,y_1,x_2,y_2,...,x_r,y_r\rangle$ and let
$M/Z=C_{G/Z}(TZ/Z)$. We apply the Three Subgroup Lemma for $T$,
$T$ and $M$ we get that $$[T,T,M]\subseteq [T,M,T]=[M,T,T]=1$$
which means that $M\leq C_G(T')$. Since $G'=T'C_{G'}(P)$ we have
that $C_G(G')=C_G(T')\cap C_G(C_{G'}(P))$. It is clear that $P\leq
C_G(C_{G'}(P))$ and so $$C_G(G')\cap P=C_G(T')\cap P\geq M\cap
P.$$

Let us denote $\hat{G'}$, $\hat{T}$, $\hat{M}$ and $\hat{P}$ the
images of $G'$, $T$, $M$ and $P$ in the factor group $G/Z$.
Applying the second statement of Lemma \ref{ketall} for
$\hat{K}=\hat{P}$ and $\hat{H}=\hat{T}$, and using that $Z\leq M$,
one gets that $$|P:P\cap \cgg|\leq|P:P\cap
M|=|PM:M|=|\hat{P}\hat{M}:\hat{M}|=$$ $$=|\hat{P}:\hat{P}\cap
\hat{M}|\leq |\hat{G'}:\hat{P}\cap\hat{G'}|^{2r}\leq (n_p)^{2r}.$$
where $n_p$ denotes the $p$-part of $|G':Z|$.

 Let
$P_1$,$P_2$,...,$P_t$ be the unique Sylow subgroups of $D$
corresponding to the prime divisors $p_1$,$p_2$,...,$p_t$ of
$|D|$. By embedding the Sylow subgroups of $\cgg$ into Sylow-$p$
subgroups of $D$ we obtain that
$$|D:\cgg|=\prod_{1\leq i\leq t}|P_i:P_i\cap \cgg|\leq
\prod_{1\leq i\leq t}n_{p_i}^{2r}=|G':Z|^{2r}.$$
\end{proof}

\begin{proof}[Proof of Theorem 3]Let $D$ be as in Lemma \ref{bilin}. Since
$$|G:D|=|G/Z:C_{G/Z}(G'Z/Z)|$$
it follows by Lemma \ref{kci}, Lemma \ref{szivas} and Corollary
\ref{also} that
$$\gzz=|G:D||D:\cgg||\cgg:\zzg|\leq |G'/Z|^{4r}.$$
\end{proof}

\begin{lem}\label{homom}
Let $G$ be a finite group and let $x$ be an element of $G$. Then
the map $a\to [a,x]$ is a homomorphism from $\cgg$ to $G'$.
\end{lem}

\begin{proof}
Assume that $a,b\in\cgg$. Then
$$[ab,x]=[a,x]^b[b,x]=[a,x][b,x].$$
\end{proof}

\begin{lem}\label{pl1}
Let $G$ be a finite $p-group$ group with $Z=G'\cap\zg$ and
$\mathrm{rk}(G'/Z)=r$. Then $\mathrm{rk}(\cgg:\zzg)\leq r^2$.
\end{lem}

\begin{proof}
We import the notation and calculations from the proof of Lemma
\ref{also}. Since $G$ is a $p$-group we have that $P=C_G(G')$,
$P\cap G'=\mathbf{Z}(G')$, $P\cap\zzg=\zzg$ and $P\cap Z=Z$.
Recall that there are elements $x_1,x_2,...,x_l$ with
$l\leq\mathrm{rk}(\cgg\cap G'/Z)\leq r$ such that if $T=\langle
x_1,x_2,\dots,x_l\rangle$ and $M/Z=C_{G/Z}(TZ/Z)$ then $M\cap
P\leq\zzg\cap P$ or equivalently $$M\cap\cgg\leq\zzg.$$ For each
$1\leq i\leq l$ let $f_i:\cgg\to G'/Z$ be the map given by
$$f_i(x)=[x,x_i]Z.$$ Lemma \ref{homom} implies that $f_i$ is a
homomorphism for all $1\leq i\leq l$ and we have that
$$\bigcap_{1\leq i\leq l}\mathrm{ker}(f_i)=M\cap\cgg\leq\zzg.$$ It
follows that the map $x\to(f_1(x),f_2(x),\dots,f_l(x))$ embeds
$\cgg/(M\cap\cgg)$ into the $l$-th direct power of $G'/Z$.
Consequently
$$\mathrm{rk}(\cgg/\zzg)\leq\mathrm{rk}(\cgg/(M\cap\cgg))\leq\mathrm{rk}((G'/Z)^l)\leq
lr\leq r^2.$$
\end{proof}

\begin{lem}\label{pl2}
Let $G$ be a finite $p$-group, $D=\{g\in G\mid [g,G']\subseteq
\zg\}$, $Z=G'\cap \zg$ and $\mathrm{rk}(G'/Z)=r$. Then
$\mathrm{rk}(D/\cgg)\leq 2r^2$.
\end{lem}

\begin{proof}
We borrow tho notation and calculations from Lemma \ref{szivas}.
Since $G$ is a $p$-group we have that $D=P$. Recall that we
constructed elements $$x_1,y_1,s_2,y_2,\dots,x_r,y_r\in G$$ where
$r=\mathrm{rk}(G'/Z)$ such that if $$T=\langle
x_1,y_1,x_2,y_2,\dots,x_r,y_r\rangle$$ and $$M/Z=C_{G/Z}(TZ/Z)$$
then

\begin{equation}\label{eq3}
\cgg=\cgg\cap D\geq M\cap D.
\end{equation}

Note that $D/\mathbf{Z}(G)$ is the centralizer of the commutator
subgroup in $G/\mathbf{Z}(G)$. It follows by Lemma \ref{homom}
that for any fixed element $x\in G$ the map $f_x(a)=[x,a]Z$ is a
homomorphism on $D$. That the map
$$a\to(f_{x_1}(a),f_{y_1}(a),f_{x_2}(a),f_{y_2}(a),\dots,f_{x_r}(a),f_{y_r}(a))$$
embeds $D/(M\cap D)$ into the $2r$-th direct power of $G'/Z$.
Using (\ref{eq3}) we obtain that $\mathrm{rk}(D/\cgg)\geq 2r^2$.
\end{proof}

We will need the following Theorem from \cite{ls} (see Proposition
16, p.363).

\begin{thm}\label{autp} Let $P$ be a finite $p$-group of rank $r$ and $Q$ a
$p$-subgroup of $\mathrm{Aut}(P)$. Then the rank of $Q$ is at most
$\frac{1}{2}(5r^2-r)$ if $p$ is odd, at most $\frac{1}{2}(7r^2-r)$
if $p=2$.
\end{thm}

\begin{proof}[Proof of Theorem 4]
Let $D=\{g\in G\mid [g,G']\subseteq \zg\}$ nad $z=G'\cap\zg$ as
usual. We have that
$$\mathrm{rk}(G/\mathbf{Z}_2(G))\leq\mathrm{rk}(G/D)\mathrm{rk}(D/C_G(G'))\mathrm{rk}(C_G(G')/\mathbf{Z}_2(G))$$
Considering the action of $G/\zg$ on $G'/Z$ we get that the group
$G/D$ is embedded into the automorphism group of $G'/Z$. According
to Theorem \ref{autp} we have that $\mathrm{rk}(G/D)\leq
\frac{1}{2}(7r^2-r)$ where $r=\mathrm{rk}(G'/Z)$. Now together with
Lemma \ref{pl1} and Lemma \ref{pl2} we get that
$$\mathrm{rk}(G/\zzg)\leq \frac{1}{2}(13r^2-r).$$
\end{proof}

\end{document}